\documentclass{article}
\usepackage[utf8]{inputenc}
\usepackage[utf8]{inputenc}
\usepackage{amsmath}
\usepackage{amsfonts}
\usepackage{amssymb}
\usepackage{amsthm}
\usepackage{geometry}
\usepackage{enumerate}
\usepackage{bbm}
\usepackage{url}
\usepackage[usenames,dvipsnames]{color}
\usepackage{algorithm, algorithmic}

\usepackage[colorlinks=false,,linktoc=allbookmarksopen=true,linktocpage,pdftex]{hyperref}
\usepackage[pdftex]{bookmark}

\def\rompar(#1){\textup(#1\textup)}    

\def\xexp(#1){e^{#1}}
\newcommand\ceil[1]{\lceil#1\rceil}

\newcommand{\refT}[1]{Theorem~\ref{#1}}

\newcommand{\refL}[1]{Lemma~\ref{#1}}

\newcommand{\refS}[1]{Section~\ref{#1}}
\newcommand{\refApp}[1]{Appendix~\ref{#1}}

\newcommand{\refCon}[1]{Conjecture~\ref{#1}}
\newcommand{\indic}[1]{\mathbbm{1}_{\{{#1}\}}}

\newcommand{\cPr}{\mathbb{P}}
\newcommand{\E}{\mathbb{E}}
\newcommand{\Bin}{\operatorname{Bin}}
\newcommand{\Var}{\operatorname{Var}}

\newcommand{\cL}{C}
\newcommand{\cQm}{\mathcal{F}_M}
\newcommand{\cQe}{\mathcal{F}_E}
\newcommand{\cxi}{\xi}

\newcommand{\cF}{\mathcal{F}}

\newtheorem{theorem}{Theorem}[]

\newtheorem{remark}[theorem]{Remark}
\newtheorem{lemma}[theorem]{Lemma}
\newtheorem{conjecture}[theorem]{Conjecture}

\title{Short rainbow cycles for families of matchings and triangles}

\date{November 20, 2022; revised March 31, 2024}
\author{He Guo\thanks{Faculty of Mathematics, Technion, Haifa~32000, Israel. E-mail: {\tt hguo@campus.technion.ac.il}.}}
\begin{document}

\maketitle

\begin{abstract}
A generalization of the famous Caccetta--H\"aggkvist conjecture, suggested by Aharoni~\cite{ADH2019}, is that any family $\mathcal{F}=(F_1, \ldots,F_n)$ of sets of edges in~$K_n$, each of size~$k$, has a rainbow cycle of length at most $\lceil \frac{n}{k}\rceil$. In~\cite{AharoniGuo, ABCGZ2022} it was shown that asymptotically this can be improved to~$O(\log n)$ if all sets are matchings of size 2, or all are triangles. We show that the same is true in the mixed case, i.e., if each~$F_i$ is either a matching of size 2 or a triangle. We also study the case that each~$F_i$ is a matching of size 2 or a single edge, or each~$F_i$ is a triangle or a single edge, and in each of these cases we determine the threshold proportion between the types, beyond which the rainbow girth goes from linear to logarithmic.
\end{abstract}

\section{Introduction}
The \emph{directed girth}~$dgirth(G)$ of a digraph~$G$ is the minimal length of a directed cycle in it ($\infty$  if there is no directed cycle). 
A famous conjecture of  Caccetta and H\"{a}ggkvist  \cite{CaccettaHaggkvist} (below --- CHC) states that any digraph $G$ on~$n$ vertices satisfies $dgirth(G) \le \lceil \frac{n}{\delta^+(G)}\rceil$, where $\delta^+(G)$ is the minimum out-degree of~$G$.  See~\cite{ChvatalSzemeredi, Hamidoune, HoangReed, Nishimura} for progress on this problem. In particular it has been shown that 

(a) The CHC is true if $n\ge 2\delta^+(G)^2-3\delta^+(G)+1$~\cite{Shen1},  and

(b)  $dgirth(G)\le n/\delta^+(G)+73$ for all~$G$~\cite{Shen2}.

In~\cite{ADH2019} a possible generalization of CHC was suggested. 
Given a family $\cF=(F_1, \ldots,F_m)$ of sets of edges, 
a set~$F$ of edges is said to be {\em rainbow} for~$\cF$ if each of the edges in~$F$ is taken from a distinct~$F_i$ (if  the sets $F_i$ are disjoint, this means that $|F\cap F_i|\le 1$ for each~$i$).  The {\em rainbow girth}~$rgirth(\cF)$ of~$\cF$ is the minimal length of a rainbow cycle with respect to~$\cF$.

\begin{conjecture}\label{con:rainbow}
For any family $\cF=(F_1, \ldots,F_n)$ of subsets of~$E(K_n)$ such that $|F_i|=k$ for each~$1\le i\le n$, we have $rgirth(\cF)\le \ceil{n/k}$. 
\end{conjecture}
We may clearly assume that the sets $F_i$ are disjoint, since otherwise there is a rainbow digon, meaning that the rainbow girth is 2. 
Given~$\cF=(F_1, \ldots,F_m)$ with $F_i\neq\emptyset$ for each $i\in [m]$ and $\cup_{i=1}^m F_i=E(G)$ for some graph~$G$, we shall refer to~$\cF$ as an \emph{edge coloring} of~$G$, the indices $i\in [m]$ as \emph{colors}, the sets $F_i$ as {\em color classes}, and $rgirth(\cF)$ as the \emph{rainbow girth} of~$G$ with respect to the edge coloring~$\cF$. 

Devos et. al.~\cite{DDFGGHMM2021} proved \refCon{con:rainbow} for $k=2$. In~\cite{ABCGZ2022} a stronger version of the conjecture was proved, in which the sets~$F_i$ are of size $1$ or $2$. 
In~\cite{HS2022} it was shown that the order of magnitude is right: there exists a constant~$C>0$ such that for any~$k, n$ and~$\cF$ satisfying the assumption of~\refCon{con:rainbow}, we have $rgirth(\cF)\le Cn/k$.
An  explanation  why~\refCon{con:rainbow} implies the CHC can be found in~\cite{AharoniGuo} and~\cite{ABCGZ2022}.

All known extreme examples for~\refCon{con:rainbow} are obtained from those of the CHC, taking the color classes as stars.
This suggests looking at  the  case when the sets of edges are not stars, and trying to improve the upper bound on the girth.
Indeed, in~\cite{AharoniGuo}, it is proved that if an $n$-vertex graph is edge-colored by~$n$ colors such that each color class is a matching of size 2, then the rainbow girth is $O(\log n)$, asymptotically improving the conclusion of the conjecture.

A set of edges  not containing a matching of size 2 is either a star or a triangle, hence the next interesting case is that of families of triangles. In~\cite{ABCGZ2022}, it is proved that a family of $n$ triangles in $K_n$ has rainbow girth~$O(\log n)$. Furthermore, it was shown there that~$\log n$ is the right order of magnitude: an $n$-vertex graph is constructed, consisting of~$n$ edge-disjoint triangles whose rainbow girth is $\Omega(\log n)$.

In this note we fine-tune the above results, 
by showing that the rainbow girth is~$O(\log n)$ in the mixed case that each~$F_i$ is either a matching of size 2 or a triangle. We also study the case that each~$F_i$ is a matching of size 2 or a single edge, or each~$F_i$ is a triangle or a single edge, and in each of these cases we determine the threshold proportion between the types, beyond which the rainbow girth goes from linear to logarithmic.

\section{Graph theoretical and probabilistic tools}\label{sec:tools}

As in~\cite{HS2022, AharoniGuo},
a key  ingredient in the proofs is  a result by Bollob\'as and Szemer\'edi~\cite{BS2002} on the girth of sparse graphs.

\begin{theorem}\label{thm:girthsparsegraph}
For $N\ge 4$ and $k\ge 2$, every $N$-vertex graph with $N+k$ edges has girth at most
\[\frac{2(N+k)}{3 k}(\log_2 k+\log_2\log_2 k+4). \]
\end{theorem}

We shall use two well-known concentration inequalities.
\begin{theorem}[Chernoff]\label{thm:chernoff}
Let $X$ be a binomial random variable $\Bin(n,p)$. 
For any $t\ge 0$, we have
\[\cPr(X\ge \E X + t)\le \exp\Big(-\frac{t^2}{2(\E X+t/3)}\Big). \]
\end{theorem}

\begin{theorem}[Chebyshev]\label{thm:chebyshev}
Let $X$ be a random variable. 
For any $t>0$, we have
\[\cPr(|X -\E X|\ge t)
\le \frac{\Var X}{t^2}, \]
where $\Var X$ is the variance of $X$.
\end{theorem}

\section{Main results}

\subsection{Matchings and single edges}
\begin{theorem}\label{thm:weaker}
    For any~$\alpha>1/2$, there exists $C$ such that for any edge coloring $\cF=(F_1,\dots, F_n)$ of an $n$-vertex graph~$G$ with $F_i\neq\emptyset$ for each $i\in [n]$, if at least $\alpha n$ color classes in~$\cF$ are matchings of size 2, then $rgirth(\cF)\le C\log n$. 
\end{theorem}

We shall need a slightly stronger result, which allows the number of size-two matchings to be 
less than $\alpha n$ and the total number of color classes to be less than~$n$.

\begin{theorem}\label{thm:main}
For any~$\alpha>1/2$, there exist  $\cxi = \cxi(\alpha)>0$ and $C=C(\alpha)$ such that 
the following holds. Let $G$ be an $n$-vertex graph and~$\cF=(F_1,\dots, F_m)$ be an edge coloring of $G$. 
If~$\cF= \cQm 	\sqcup \cQe$, where 
\begin{enumerate}
  \item   every $F_i\in \cQm$ is a matching of size 2, 
  \item   every $F_i\in \cQe$ is a single edge,
  \item $|\cQm|\ge (\alpha-\cxi)n$ and $|\cQe|\ge(1-\alpha-\cxi)n$,
\end{enumerate}
then $rgirth(\cF)\le C\log n$.
\end{theorem}

\refT{thm:main} will follow from~\refT{thm:girthsparsegraph}, and the following:

 \begin{theorem}\label{thm:maintechnical}
 For any~$\alpha > 1/2$, there exist $\beta,c>0$ such that for any large enough $n$, given an $n$-vertex graph~$G$ and an edge coloring of~$G$ satisfying the assumption in~\refT{thm:main}, there exists  a 
 subset~$S$ of $V(G)$ of size at most $\beta n$ containing a rainbow edge set of size at least $(\beta+c)n$.
 \end{theorem}

   Once this is proved, \refT{thm:main} follows by applying \refT{thm:girthsparsegraph} with $N=\beta n$ and $k=c n$.

In~\refS{subsubsec:simple}, we give a simplified proof of~\refT{thm:main}. But we still think the proof of~\refT{thm:main} via~\refT{thm:maintechnical} involves some interesting techniques.

 The idea used to prove~\refT{thm:maintechnical} is choosing a random subset~$S$ of~$V(G)$ and considering the induced subgraph~$G[S]$. The crux of the argument is that the expected number of vertices $\E |S|$ is less than the expected number of rainbow edges in~$G[S]$, and their difference is linear in~$n$. Furthermore,  these two random numbers are concentrated around their expectations, which follows from the concentration inequalities in~\refS{sec:tools}.

Theorem 2.8 of~\cite{AharoniGuo} states that for any~$\gamma>3\sqrt{6}/8\approx 0.9186$ and $n$-vertex graph with an edge coloring with at least~$\gamma n$ colors, if each color class is a matching of size 2, then the rainbow girth is~$O(\log n)$.
For completeness, we add an alternative proof of this statement in the~\refApp{app:App}, which was observed by Michael Krivelevich.\footnote{The bound on $\gamma$ is a bit weaker, but enough for our later use.}
Note that in~\refT{thm:main} we can take~$\cxi(\alpha)>0$ arbitrarily small. In particular, if $\alpha> 3\sqrt{6}/8$, we can guarantee that $\alpha-\cxi > 3\sqrt{6}/8$. Therefore here we may assume that $\alpha\le 3\sqrt{6}/8$. 
Again since~$\cxi$ can be chosen arbitrarily small, we may assume that $\min\{\alpha-\cxi,1-\alpha-\cxi \}\ge 1/40$ so that $\min\{|\cQm|, |\cQe|\}\ge n/40$. And without loss of generality, we may assume that $\max\{|\cQm|, |\cQe|\}\le n$.

First we claim that for  $p\in (0,1)$ close enough to 1, we have  
\begin{equation}\label{eq:inqtotal}
    \alpha(2p^2-p^4)+(1-\alpha)p^2>p. 
\end{equation}
In fact, when $\alpha = 1/2+\delta$ for some $\delta>0$, the above is equivalent to 
\[(1/2+\delta)(2p-p^3)+(1/2-\delta)p>1. \]
Writing $p=1-\tau$, we have
\begin{align*}
    (1/2+\delta)(2p-p^3)+(1/2-\delta)p &= \frac{3}{2}p+\delta p -(\frac{1}{2}+\delta)p^3\\
   & = \frac{3}{2}(1-\tau)+\delta(1-\tau)-(\frac{1}{2}+\delta)(1-\tau)^3\\
    &=1+2\delta \tau +\tau^2(-3/2-3\delta+\tau/2+\delta\tau),
\end{align*}
which is greater than 1 for $\tau = \tau(\delta)>0$ small enough.

For any $p\in (0,1)$ close enough to 1  satisfying~\eqref{eq:inqtotal}, there exist constants $\epsilon(\alpha,p),\cxi(\alpha,p)>0$ small enough so that
\begin{equation}\label{eq:inqtrue}
    (1-\epsilon)(\alpha-\cxi)(2p^2-p^4)+(1-\epsilon)(1-\alpha-\cxi)p^2>p.
\end{equation}

Fix $1/2\le p <1$   and $\epsilon,\cxi>0$ that satisfy~\eqref{eq:inqtrue}.

A vertex~$v$ is called a \emph{heavy} 
vertex if there are at least $(\epsilon^2/10^6)n$ many rainbow edges, i.e., edges in~$(\cup\cQm)\cup (\cup\cQe)$, incident to it.  Let~$D$ be the set of all heavy vertices. Then
\begin{equation}\label{eq:boundonD}
    |D|\le (2\cdot 2|\cQm|+2|\cQe|)/((\epsilon^2/10^6)n )\le 6n/((\epsilon^2/10^6)n )\le  10^7/\epsilon^2.
\end{equation}
We construct a random vertex set
\[S:=D\cup (V\setminus D)_p \]
i.e., $S$ contains the set~$D$ of heavy vertices and includes each vertex of~$V\setminus D$ independently with probability~$p$.

\begin{lemma}\label{lemma:S}
With high probability\footnote{An event holds with high probability if the probability of that event tends to 1 as~$n$ tends to infinity.}, $|S|\le np+ n^{2/3}$.
\end{lemma}
\begin{proof}
Note that by construction, $|S|$ has the same probability distribution as $|D|+\Bin(n-|D|, p)$. Set $Z=\Bin(n-|D|, p)$. Applying Chernoff's bound (\refT{thm:chernoff}) and~\eqref{eq:boundonD}, we have
\[\cPr(Z \ge np+n^{2/3}-|D| )\le \cPr(Z\ge \E Z +n^{2/3}-|D|) \le  \exp(-n^{\Omega(1)}). \]
Therefore with high probability, $|S|\le |D|+Z\le np+ n^{2/3}$.
\end{proof}

\begin{lemma}
With probability at least 0.9, the number of color classes in~$\cQm$ that have at least one edge contained in~$S$ is at least $(1-\epsilon)|\cQm|\cdot (2p^2-p^4)$.
\end{lemma}
\begin{proof}
For $F_i\in \cQm$, let 
\[X_i:=\indic{\text{at least one edge in $F_i$ is contained in $S$} }\] 
be the indicator random variable that some edge~$e$ in~$F_i$ is contained in~$S$.
Since each vertex is included in~$S$ independently with probability at least~$p$ and~$X_i$ is an increasing random variable with respect to the probability that a vertex is included in~$S$, by inclusion-exclusion we have
\[ \E X_i \ge 2p^2-p^4. \] 
Let 
\begin{equation}\label{eq:def:X}
    X:=\sum_{F_i\in \cQm} X_i.
\end{equation}
We have 
\begin{equation}\label{eq:EX}
 \E X \ge |\cQm|\cdot (2p^2-p^4).
\end{equation}
To prove the lemma, we shall apply Chebyshev's inequality (\refT{thm:chebyshev}). For this purpose we have to estimate $\Var X$. With a look at~\eqref{eq:def:X}, we have  
\begin{align}
    \Var X = \E X^2-(\E X)^2= \sum_{F_i,F_j\in \cQm} (\E X_iX_j- \E X_i\E X_j).
\end{align}

Note that if the matchings of the colors $i,j$ are vertex-disjoint, then $X_i$ and $X_j$ are independent and $\E X_iX_j- \E X_i\E X_j =0$.

Since the matchings in $\cQm$ are disjoint, for every $F_i\in \cQm$,  at most $6=\binom{4}{2}$  matchings $F_j\in\cQm$ can have an edge contained in  $\bigcup F_i$. This means that there exist at most $2\cdot 6n$ pairs $(F_i, F_j)$ such that an edge from $F_j$ is contained in $\bigcup F_i$, or vice versa. Thus the contribution of such pairs to $\Var X$ is at most $12n$.

Apart from vertex-disjointness, there are  two more possible forms of $ F_i \cup F_j$:

I. three connected components: one 2-path and two disjoint edges, or 

II.  two vertex-disjoint 2-paths.

Examine Case I. 
Let $F_i=\{a,b\}$, where $a=\{x,y\},~b=\{u,v\}$, and let $F_j=\{c,d\}$, where $c=\{x,z\}$ and $d=\{s,t\}$. 

If $x$ is a heavy vertex, then $\E X_iX_j-\E X_i\E X_j=0$.  Indeed, since the events $\{a\subseteq S\}=\{y\in S\}$, $\{b\subseteq S\}$, $\{c\subseteq S\}=\{z\in S\}$, and $\{d\subseteq S\}$ are mutually independent, we have
\begin{align*}
&\E X_i \E X_j \\
=&\Big(\cPr(a\subseteq S \text{ and }b\not\subseteq S)+\cPr(a\not\subseteq S \text{ and }b\subseteq S)+\cPr(a\subseteq S \text{ and }b\subseteq S)\Big)\\
&\cdot\Big(\cPr(c\subseteq S \text{ and }d\not\subseteq S)+\cPr(c\not\subseteq S \text{ and }d\subseteq S)+\cPr(c\subseteq S \text{ and }d\subseteq S)\Big)\\
=&\Big(\cPr(a\subseteq S)\cPr(b\not\subseteq S)+\cPr(a\not\subseteq S)\cPr(b\subseteq S)+\cPr(a\subseteq S)\cPr(b\subseteq S)\Big)\\
&\cdot \Big(\cPr(c\subseteq S)\cPr(d\not\subseteq S)+\cPr(c\not\subseteq S)\cPr(d\subseteq S)+\cPr(c\subseteq S)\cPr(d\subseteq S)\Big)\\
=&\cPr(\text{$a,c \subseteq S$ and $b,d \not \subseteq S$})+\cPr(\text{$a,d \subseteq S$ and  $b,c \not \subseteq S$})+\cPr(\text{$b,c \subseteq S$ and $a,d \not \subseteq S$})\\
&+\cPr(\text{$b,d \subseteq S$ and $a,c \not \subseteq S$})
+\cPr(\text{$a,b,c\subseteq S$ and $d\not\subseteq S$})+\cPr(\text{$a,b,d\subseteq S$ and $c\not\subseteq S$})\\
&+\cPr(\text{$a,c,d\subseteq S$ and $b\not\subseteq S$})+\cPr(\text{$b,c,d\subseteq S$ and $a\not\subseteq S$})+\cPr(\text{$a,b,c,d\subseteq S$})\\
=&\E X_iX_j.
\end{align*}

If $x$ is not a heavy vertex, the contribution to $\Var X$ for such $(F_i,F_j)$ is at most $4\cdot n\cdot \frac{\epsilon^2 n}{10^6}$, since there are at most $n$ ways to choose $F_i$, four ways to choose $x$, and at most $\frac{\epsilon^2 n}{10^6}$ ways to choose $F_j$.

In case II, let $F_i=\{a,b\}$, where $a=\{x,y\},~b=\{u,v\}$, and let $F_j=\{c,d\}$, where $c=\{x,z\}$ and $d=\{u,t\}$. 

If both $x$ and $u$ are heavy vertices, then the events $\{a\subseteq S\}=\{y\in S\}$, $\{b\subseteq S\}=\{v\in S\}$, $\{c\subseteq S\}=\{z\in S\}$, and $\{d\subseteq S\}=\{t\in S\}$ are mutually independent.
Similarly to case I, we have $\E X_iX_j-\E X_i\E X_j=0$.  Otherwise the contribution to $\Var X$ for such $(F_i,F_j)$ is at most $4\cdot n\cdot \frac{\epsilon^2 n}{10^6}$.

Summing, for $n$ large enough (so that $12\le \frac{2\epsilon^2}{10^6}n$) we have 
\[ \Var X\le 12n+2\cdot4\cdot n\cdot \frac{\epsilon^2 n}{10^6}\le \frac{\epsilon^2 n^2}{10^5}.   \]

Since $1/2\le p<1$, we have $2p^2-p^4\ge 7/16$.
Applying Chebyshev's inequality and using the assumption $|\cQm|\ge n/40$, we have
\begin{align*}
    \cPr(X \le (1-\epsilon)|\cQm|\cdot (2p^2-p^4))&\le \cPr(X\le \E X -\epsilon|\cQm|\cdot (2p^2-p^4)) \\
    &\le \frac{\Var X}{(\epsilon|\cQm|\cdot (2p^2-p^4))^2} \le \frac{\epsilon^2 n^2}{10^5 \epsilon^2(n/40)^2(7/16)^2}\le 1/10, 
\end{align*}
which completes the proof.
\end{proof}

\begin{lemma}\label{lemma:Y}
With probability at least 0.9, the number of color classes in~$\cQe$ that are contained in~$S$ is at least $(1-\epsilon)|\cQe|\cdot p^2$.
\end{lemma}
\begin{proof}
For $F_i\in \cQe$, let 
\[Y_i:=\indic{\text{$F_i\subseteq S$}}\]
be the indicator random variable that the edge of color~$i$ is contained in $S$.
Then $\E Y_i\ge p^2$.
Let \[Y:=\sum_{F_i\in \cQe}Y_i.\] We have 
\begin{equation}\label{eq:EY}
    \E Y\ge |\cQe|\cdot p^2.
\end{equation}
To prove the lemma, we shall apply Chebyshev's inequality (\refT{thm:chebyshev}). For this purpose we have to estimate $\Var X$. We have  
\begin{align}
    \Var Y = \E Y^2-(\E Y)^2= \sum_{F_i,F_j\in \cQe} (\E Y_iY_j- \E Y_i\E Y_j).
\end{align}
Note that if the edges of colors~$i,j$ are vertex-disjoint, then $Y_i$ and $Y_j$ are independent and $\E Y_iY_j- \E Y_i\E Y_j =0$. 

Furthermore, if the edges of two distinct color classes $F_i,F_j\in \cQe$, say $uv$ and $uw$, intersect at a heavy vertex~$u$ that is in~$D$, then
$\E Y_iY_j- \E Y_i\E Y_j = \cPr(v,w\in S)-\cPr(v\in S)\cPr(w\in S)=0 $.

Therefore besides the case $F_i=F_j$, the non-zero contribution to~$\E Y$ can only come from $F_i,F_j$ that the edges intersect at a non-heavy vertex. Note that by definition, such a non-heavy vertex is incident to at most~$(\epsilon^2/10^6) n$ edges in~$\cup\cQe$. Therefore there are at most $2\cdot |\cQe|\cdot 2 \cdot (\epsilon^2/10^6) n$ such pairs~$(F_i,F_j)$.

Therefore using the assumption that $|\cQe|\ge n/40$, for~$n$ large enough (so that $1\le \frac{4\epsilon^2}{10^6}n$) we have
\[\Var Y\le |\cQe|+4|\cQe| (\epsilon^2/10^6) n\le \frac{8\cdot 40 \epsilon^2}{10^6}|\cQe|\cdot \frac{n}{40}\le \frac{\epsilon^2}{10^3} |\cQe|^2.  \]
Therefore by Chebyshev's inequality and $p\ge 1/2$, we have
\[\cPr(Y\le |\cQe|\cdot p^2- \epsilon |\cQe|\cdot p^2) \le \cPr(Y\le \E Y- \epsilon |\cQe|\cdot p^2)\le \frac{\Var Y}{(\epsilon |\cQe|\cdot p^2)^2}\le  \frac{\epsilon^2|\cQe|^2}{10^3\cdot \epsilon^2 |\cQe|^2\cdot (1/2)^4}\le 1/10, \]
which completes the proof.
\end{proof}

\begin{proof}[Proof of~\refT{thm:maintechnical}]
Combining~\refL{lemma:S}--\ref{lemma:Y} and taking a union bound, we know that with positive probability (at least 1/2), all of the following hold:
\[|S|\le np+n^{2/3},\quad  X\ge (1-\epsilon)|\cQm|\cdot(2p^2-p^4),\quad \text{and } \quad Y\ge (1-\epsilon)|\cQe|\cdot p^2. \]
Therefore there is some~$S$ such that all of the above hold.
Then the number of rainbow edges contained in~$S$ is at least 
\begin{align*}
    X+Y &\ge (1-\epsilon)|\cQm|\cdot(2p^2-p^4)+(1-\epsilon)|\cQe|\cdot p^2\\
       &\ge (1-\epsilon)(\alpha-\cxi) n (2p^2-p^4) +(1-\epsilon)(1-\alpha-\cxi)n\cdot p^2.
\end{align*}
With a look at~\eqref{eq:inqtrue} and setting $3c:=(1-\epsilon)(\alpha-\cxi)(2p^2-p^4)+(1-\epsilon)(1-\alpha-\cxi)p^2-p>0 $, for~$n$ large enough we have that $|S|\le \beta n$ for $\beta:=p+c$ and the number of rainbow edges contained in~$S$ is at least $(\beta+c)n=\big( (p+3c)-c \big)n$. Therefore this completes the proof of~\refT{thm:maintechnical}. 
\end{proof}

\subsubsection{A simplified proof of~\refT{thm:main}}\label{subsubsec:simple}
Indeed, \refT{thm:main} will follow from the following theorem and~\refT{thm:girthsparsegraph} by setting $N=|S|\le n$ and $k=cn$.
\begin{theorem}
     For any~$\alpha > 1/2$, there exist $c>0$ such that for any large enough $n$, given an $n$-vertex graph~$G$ and an edge coloring of~$G$ satisfying the assumption in~\refT{thm:main}, there exists  a 
 subset~$S$ of $V(G)$ of such that the size of a rainbow edge set contained in~$S$ is at least $cn$ larger than~$|S|$.
\end{theorem}
\begin{proof}
    We take $p\in (0,1)$ and $\cxi>0$ satisfying~\eqref{eq:inqtrue} and set 
\[ c:=  (\alpha-\cxi)(2p^2-p^4)+(1-\alpha-\cxi)p^2-p>0.  \]    
 And we included each vertex of~$G$ into~$S$ independently with probability~$p$.
 Let $r_S$ be the largest size of a rainbow edge set contained in $S$.
By linearity of expectations and following similarly as the arguments in~\eqref{eq:EX} and~\eqref{eq:EY}, we have
\begin{align*}
    &\E(r_S-|S|)\\
    =& |\cQm|(2p^2-p^4)+|\cQe|p^2-pn\\
    \ge & \Big((\alpha-\cxi)(2p^2-p^4)+(1-\alpha-\cxi)p^2-p\Big)n\ge cn.
\end{align*}
 Therefore there exists an instance of $S$ such that the conclusion of the theorem holds.   
\end{proof}

\subsubsection{Sharpness of the condition~$\alpha>1/2$}
To get the logarithmic in~$n$ bound on the rainbow girth, it is necessary to assume $\alpha > 1/2$ in~\refT{thm:weaker} and~\refT{thm:main}. The following $n$-vertex graph~$F$ with~$n/2$ matchings of size 2 in~$\cQm$ and $n/2$ single edges in~$\cQe$ (so that $\alpha =1/2$) has rainbow girth linear in~$n$. 

For simplicity, we may assume that~$n \ge 8$ is divisible by 4.
The vertices of~$F$ are $v_{i,j}$ for $i=1,\dots, n/4$ and $j=1,\dots, 4$. For each $1\le i\le n/4$, the four vertices~ $v_{i,j}$ for $1\le j\le 4$ form a 4-cycle: $\{v_{i,1}v_{i,2}, v_{i,3}v_{i,4}\}$ is a matching of size 2 from~$\cQm$ and $v_{i,2}v_{i,3},v_{i,4}v_{i,1}$ are two edges in two colors from~$\cQe$.
And $\{v_{i,3}v_{i+1,2}, v_{i,4}v_{i+1,1}\}$ (with the first subscripts module $n/4$) for each $1\le i\le n/4$ is a matching of size 2 from~$\cQm$. It can be verified that the $n$-vertex graph~$F$ satisfies the assumption but has rainbow girth~$n/2$.


\subsection{Triangles and single edges}
\begin{theorem}\label{thm:weakertrianglesedges}
    For any constant~$\alpha>0$,
   if an $n$-vertex graph~$G$ and edge coloring of~$G$ satisfying that at least $\alpha n$ color classes consisting of a triangle and at least $(1-\alpha)n$ color classes consisting of a single edge, then the rainbow girth is at most $\cL \log n$ for some constant $\cL(\alpha)>0$. 
\end{theorem}
We can prove a slightly stronger result, which allows the total number of color classes to be less than~$n$.

\begin{theorem}\label{thm:trianglesedges}
    For any constants~$\alpha > 0$ and $\cxi <\alpha/3$,
   if an $n$-vertex graph~$G$ and edge coloring of~$G$ satisfying that at least $(\alpha-\cxi) n$ color classes consisting of a triangle and at least $(1-\alpha-\cxi)n$ color classes consisting of a single edge, then the rainbow girth is at most $\cL \log n$ for some constant $\cL(\alpha,\cxi)>0$. 
\end{theorem}
\begin{proof}
We take each of the single edge and two edges from each triangle. Then in the $n$-vertex graph, we have at least $2\cdot (\alpha -\cxi)n+(1-\alpha-\cxi)n = (1+\alpha-3\cxi)n$
    edges. Since $1+\alpha-3\cxi >1$, \refT{thm:girthsparsegraph} implies that there is a cycle of length at most $\cL\log n$ for some constant $\cL(\alpha,\cxi)>0$. If this cycle is not rainbow, we can replace two edges of the same color, which must come from the same triangle, by the other edge in the triangle to get a shorter cycle. Do it repeatedly until we obtain a rainbow cycle, which is of length at most~$\cL\log n$. This completes the proof. 
\end{proof}

\subsubsection{Sharpness of the condition $\alpha > 0$}
To get a logorithmic in~$n$ bound on the rainbow grith, we need $\alpha > 0$ in~\refT{thm:weakertrianglesedges} and~\refT{thm:trianglesedges}. Otherwise for $\alpha =0$, an $n$-cycle with~$n$ edges in distinct colors has rainbow girth~$n$.

\subsection{Matchings and triangles}

\begin{theorem}\label{thm:matchingtriangle}
There exists a constant $\cL>0$ such that for any $n$-vertex graph~$G$ and edge coloring of~$G$ with~$n$ colors, if each color class is either a matching of size 2 or a triangle, then the rainbow girth is at most $\cL\log n$. 
\end{theorem}
\begin{proof}
This result follows immediately from~\refT{thm:weaker} and~\refT{thm:weakertrianglesedges}.
If the number of triangles is at least $\alpha n$ for some $\alpha>0$, then~\refT{thm:weakertrianglesedges} implies the result. Otherwise, the number of matchings of size 2 is at least $(1-o(1))n$, where~\refT{thm:weaker} implies the result.
\end{proof}

\bigskip{\noindent\bf Acknowledgements.} 

We thank Ron Aharoni for helpful suggestions during the preparation of this note.
We thank Michael Krivelevich for pointing out an alternative proof of a theorem that we cite.
We thank Lutz Warnke for meaningful discussion inspiring the simplified proof of~\refT{thm:main}.
And thanks for the comments from the referees.

\appendix
\section{An alternative proof of rainbow girth for matchings}\label{app:App}
Michael Krivelevich pointed out the following result to the author in a personal communication.
\begin{theorem}\label{thm:new}
For any $\gamma>0$ satisfying $1-(\frac{1}{2})^{4\gamma}<\gamma$, the following holds:
    For an $n$-vertex graph $G$ with an edge coloring with at least~$\gamma n$ colors, if each color class is a matching of size 2, then the rainbow girth of $G$ is at most $C\log n$ for some $C(\gamma)>0$. 
\end{theorem}
\begin{proof}
    Let $\cQm$ be the collection of $\gamma n$ many size-2 matchings.  
    Let $H$ be the graph with vertex set $V(G)=[n]$ and edge set $\cup\cQm$.
    Independently we take one edge from each matching $F\in \cQm$ uniformly at random, and we set the chosen edge set as~$T$. Then 
    \[ |T|=\gamma n. \]
    Let $Z$ be the vertex subset $\{z\in V(G): z\not\in e \text{ for any $e\in T$}\}$. If we can show $Z\ge (1-\gamma')n$ for some constant $\gamma'<\gamma$. Then the edge set $T$ is contained in a vertex set of size at most $\gamma'n$. Then~\refT{thm:girthsparsegraph} implies there is a rainbow cycle in $T$ of length $O(\log n)$, which completes the proof. 

Note that for each vertex $v\in V(G)$, 
\[ \cPr(v\in Z)=\Big(\frac{1}{2}\Big)^{{d_H}(v)}, \]
i.e., none of edges in $H$ that are incident to $v$ is chosen into $T$.
    Therefore 
    \[\E |Z| =\sum_{v\in V(G)}\Big(\frac{1}{2}\Big)^{{d_H}(v)}\ge n \Big(\frac{1}{2} \Big)^{\sum_{v\in V(G)}d_H(v)/n}=n\Big(\frac{1}{2} \Big)^{4\gamma},\]
where the inequality is by the convexity of the function $f(x)=(1/2)^x$.
Therefore there exists a choice of $T$ such that 
\[  |Z| \ge n\Big(\frac{1}{2} \Big)^{4\gamma}. \]
Hence we can set $\gamma'=1-(\frac{1}{2} )^{4\gamma}$. By assumption we have $\gamma'<\gamma$, and this completes the proof.
\end{proof}
Though not necessary for proving the theorem, applying the bounded difference inequality, one can also show that $Z$ in the proof is concentrated around $\E Z$, since each choice of edge influences the result by at most 4.
\begin{remark}
    Numerical solution says for $\gamma\ge 0.922523266904828$, we have $1-(\frac{1}{2})^{4\gamma}<\gamma$ so that the conclusion of~\refT{thm:new} holds. 
\end{remark}


\small

\begin{thebibliography}{99}


\bibitem{ABCGZ2022}
R.~Aharoni, E.~Berger, M.~Chudnovsky, H. Guo, and S.~Zerbib.
\newblock Non-uniform degrees and rainbow versions of the {C}accetta-{H}\"aggkvist conjecture.
\newblock {\em SIAM J. Discrete Math.} {\bf 37} (2023), 1704--1714.


\bibitem{ADH2019}
R.~Aharoni, M.~DeVos, and R.~Holzman.
\newblock Rainbow triangles and the {C}accetta-{H}\"aggkvist conjecture.
\newblock {\em J. Graph Theory} {\bf 92} (2019), 347--360.




\bibitem{AharoniGuo}
R.~Aharoni and H.~Guo.
\newblock Rainbow cycles for families of matchings.
\newblock {\em Israel J. Math.} {\bf 256} (2023), 1--8.

\bibitem{BS2002}
B.~Bollobás and E.~Szemerédi.
\newblock Girth of sparse graphs.
\newblock {\em J. Graph Theory} {\bf 39} (2002), 194--200.







\bibitem{CaccettaHaggkvist}
L.~Caccetta and R.~H\"{a}ggkvist.
\newblock On minimal digraphs with given girth.
\newblock {\em Congress. Numer.} {\bf 21} (1978), 181--187.


\bibitem{ChvatalSzemeredi}
V.~Chv\'{a}tal and E.~Szemer\'{e}di.
\newblock Short cycles in directed graphs.
\newblock {\em J. Combin. Theory Ser. B} {\bf 35} (1983), 323--327.

\bibitem{DDFGGHMM2021}
M.~DeVos, M.~Drescher, D.~Funk, S.~González Hermosillo de~la Maza, K.~Guo,
  T.~Huynh, B.~Mohar, and A.~Montejano.
\newblock Short rainbow cycles in graphs and matroids.
\newblock {\em J. Graph Theory} {\bf 96} (2021), 192--202.







\bibitem{Hamidoune}
Y.O.~Hamidoune.
\newblock A note on minimal directed graphs with given girth.
\newblock {\em J. Combin. Theory Ser. B} {\bf 43} (1987), 343--348.

\bibitem{HoangReed}
C.T.~Ho\`ang and B.~Reed.
\newblock A note on short cycles in diagraphs.
\newblock {\em Discrete Math.} {\bf 66} (1987), 103--107.




\bibitem{HS2022}
P.~Hompe and S.~Spirkl.
\newblock  Further approximations for {A}haroni’s rainbow generalization of the {C}accetta-{H}\"aggkvist conjecture.
\newblock {\em Electron. J. Combin.} {\bf 29} (2022), \#P1.55.




\bibitem{Nishimura}
T.~Nishimura.
\newblock Short cycles in digraphs.
\newblock {\em Discrete Math.} {\bf 38} (1988), 295--298.









\bibitem{Shen1}
J.~Shen.
\newblock On the girth of digraphs.
\newblock{\em Discrete Math.} {\bf 211} (2000), 167--181.


\bibitem{Shen2}
J.~Shen.
\newblock On the Caccetta-H\"aggkvist conjecture.
\newblock{\em 	Graphs Combin.} {\bf 18} (2002), 645--654.






\end{thebibliography}
\normalsize
\end{document}